\documentclass[a4paper, oneside, reqno]{amsart}

\usepackage{hyperref}
\usepackage{amsmath, amsthm, amssymb}
\usepackage[numbers, square, sort&compress]{natbib}
\usepackage{enumitem}
\usepackage[reftex]{theoremref}
\usepackage{lipsum}

\newtheorem{lemma}{Lemma}
\newtheorem{theorem}{Theorem}

\theoremstyle{definition}
\newtheorem{definition}{Definition}

\newtheorem{example}{Example}

\newcommand{\R}{\ensuremath{\mathbb{R}}}
\newcommand{\N}{\ensuremath{\mathbb{N}}}
\DeclareMathOperator{\CBB}{CBB}
\DeclareMathOperator{\CAT}{CAT}
\DeclareMathOperator{\RCD}{RCD}
\DeclareMathOperator{\CD}{CD}

\DeclareMathOperator{\sn}{sn}

\DeclareMathOperator{\modang}{\tilde\measuredangle}
\DeclareMathOperator{\pack}{pack}
\DeclareMathOperator{\geod}{geod}

\title{Besicovitch covering on Alexandrov spaces}
\author{Maxime Marot}
\address{Christian-Albrechts-Universität zu Kiel, Heinrich-Hecht-Platz 6, 24118 Kiel}
\email{marot@math.uni-kiel.de}
\subjclass[2020]{28A75, 28A78, 53C23}
\keywords{Alexandrov space, Curvature bounded below, Directionally limited, Besicovitch covering}

\begin{document}

\begin{abstract}
We prove that finite-dimensional geodesic locally compact Alexandrov spaces with curvature bounded below (CBB),
and in particular complete Riemannian manifolds with lower bounded sectional curvature,
satisfy a Besicovitch covering theorem by showing that they are directionally limited in the sense of Federer.
\end{abstract}

\maketitle

\section{Introduction}

In the Euclidean space, it is well-known since Besicovitch \cite{besicovitchCovering1945, besicovitchCovering1946}
that the following covering theorem holds:

\begin{theorem}[Euclidean Besicovitch Covering Theorem]
There exists a universal constant $C=C(n)$ such that: for any bounded set $A\subset\R^n$ and a cover $\mathcal{F}$
of $A$ by closed balls such that every point of $A$ is the center of a ball in $\R^n$, then there exist $C$ subfamilies
$\mathcal{G}_1,\dots,\mathcal{G}_C$ of $\mathcal{F}$ each composed of pairwise disjoint balls such that
\[
A \subset \bigcup_{i=1}^C \bigsqcup_{B\in\mathcal{G}_i} B.
\]
\end{theorem}

Herbert Federer introduced the directional limitedness of a metric in \cite{federerGMT1996}
as a sufficient geometric condition on a metric space for Besicovitch-type covering theorems
(see Theorems \ref{thm:BesicovitchLemma} and \ref{thm:BesicovitchThm} below).
The precise definition is as follows:

\begin{definition}[{\cite[Definition 2.8.9]{federerGMT1996}}]
\th\label{def:DirLim}
Let $(X,d)$ be a metric space and let $A\subset X, \xi>0, 0<\eta\leq 1/3$ and $\zeta\in\mathbb{N}^*$.
We say that $d$ is \emph{directionally $(\xi,\eta,\zeta)$-limited at $A$} if the following condition holds:
If for any $a\in A$, any $E\subset A \cap B(a,\xi)\setminus\{a\}$, and if
\[
\frac{d(x,c)}{d(a,c)} \geq \eta
\]
whenever $x\in X, b,c\in E$ with $b\neq c$, $d(a,b)\geq d(a,c)$ and
\[
d(a,x) = d(a,c), \quad d(x,b) = d(a,b) - d(a,c),
\]
then $\# E \leq\zeta$.
More generally, a metric space $(X,d)$ is called \emph{directionally limited} if $d$ is $(\xi,\eta,\zeta)$-directionally limited
at $X$ for some $(\xi, \eta,\zeta)$.
\end{definition}

Intuitively, a metric is directionally limited if, from any point, only a uniformly bounded number of distinct
``directions'' can be observed by the metric at small scales.

The goal of this paper is to prove:

\begin{theorem}
\th\label{thm:CBBDirLim}
Let $(X,d)$ be an $n$-dimensional geodesic locally compact locally $\CBB(\kappa)$ space and let $0<\xi\leq\varpi^\kappa/2$
and $0<\eta\leq 1/3$.
Then there exists a positive integer $\zeta=\zeta(\kappa, n,\xi,\eta)$ such that $d$ is $(\xi, \eta,\zeta)$-directionally limited at $X$.
\end{theorem}

In particular, the conclusion also holds for complete length CBB spaces of finite dimension
\cite[Corollary 15.14]{alexanderAlexandrovGeometry2024} and complete Riemannian manifolds with sectional curvature bounded below.
In fact, the Riemannian case was previously studied in the literature by Takahashi-Yamamoto \cite{takahashiSolvability2023} under much stronger
hypotheses (complete, $\operatorname{inj}>0$ and $|\operatorname{sec}|\leq\kappa$).
A Besicovitch-type covering theorem was also directly shown by Jost-Vân Lê-Dat Tran in \cite{jostDifferentiation2020} for
complete Riemannian manifolds and balls with radii bounded by $\operatorname{inj}/4$.

Supposing now that we have fixed a finite-dimensional geodesic locally compact locally CBB space $(X,d)$,
\thref{thm:CBBDirLim} yields the covering theorems:

\begin{theorem}[Besicovitch's Covering Lemma, {\cite[Theorem 2.8.14]{federerGMT1996}}]
\th\label{thm:BesicovitchLemma}
Let $A\subset X$ and let $\mathcal{F}$ be a family of closed balls of radii uniformly bounded by a constant $\leq\xi/2$
such that each point of $A$ is the center of a ball in $\mathcal{F}$.
Then, there exists $2\zeta+1$ countable subfamilies $\mathcal{G}_1,\dots,\mathcal{G}_{2\zeta+1}$ of $\mathcal{F}$ each composed
of pairwise disjoint balls such that
\[
A \subset \bigcup_{i=1}^{2\zeta+1} \bigsqcup_{B\in\mathcal{G}_i} B.
\]
\end{theorem}

\begin{theorem}[Besicovitch's Covering Theorem, {\cite[Corollary 2.8.15]{federerGMT1996}}]
\th\label{thm:BesicovitchThm}
Let $A\subset X$ and let $\mathcal{F}$ be a family of closed balls with radii uniformly bounded by a constant $\leq \xi/2$
such that each point of $A$ is the center of a ball and
\[
\inf \{r>0 | B[a,r] \in \mathcal{F} \} = 0,\quad \forall a \in A,
\]
and let $\mu$ be a positive Borel measure on $X$ with $\mu(A)<+\infty$.
Then, for any open set $U\subset X$, there exists a countable subfamily $\mathcal{G}\subset\mathcal{F}$ of pairwise disjoint balls such that
\[
\bigsqcup_{B\in\mathcal{G}} B\subset U \quad\text{and}\quad \mu\left( (U\cap A) \setminus \bigsqcup_{B\in\mathcal{G}} B \right) = 0.
\]
\end{theorem}

Other notable implications of directional limitedness are:
\begin{itemize}
\item
Characterization of the singular part in the Lebesgue decomposition of any pair of Borel measures \cite[Theorem 2.9.15]{federerGMT1996},
\item
The set of finite Borel measures over $(X,d)$ endowed with the Carathéodory distance is complete \cite{dorschCompleteness2023},
\item
Any positive Borel measure finite on bounded sets can be recovered from its values on closed balls \cite{buetApproximate2016},
\item
Any locally finite Borel measure admits a Lebesgue differentiation theorem \cite[Theorem 5.2]{rigotDifferentiation2022}.
\end{itemize}

The finite-dimensional CBB condition is the right setting for directional limitedness in Alexandrov geometry.
The following examples inspired from \cite{buetApproximate2016} show that infinite-dimensional and $\CAT(0)$ spaces
need not be directionally limited.

\begin{example}
Let $(H,\lVert\cdot\rVert)$ be a separable Hilbert space with Hilbert basis $\mathcal{B}=(e_i)_{i\in\N}$.
The space $H$ is both $\CBB(0)$, $\CAT(0)$ and $\RCD(1, \infty)$ however it is not directionally limited.
Indeed, fix $\xi>0$ and let $E=\{\xi e_i\}_{i\in\N}$.
Pick any $a\in H$ and set $b=a+\xi e_j, c=a+\xi e_k$.
As the Hilbert norm $\lVert\cdot\rVert$ is strictly convex, the point $x$ in \thref{def:DirLim} is uniquely defined,
see \cite[2.8.9]{federerGMT1996}, by
\[
x = a + \frac{\lVert c-a \rVert}{\lVert b-a \rVert} e_j = a + \xi e_j.
\]
But
\[
\frac{d(x,c)}{d(a,c)} = \frac{\lVert \xi e_j - \xi e_k \rVert}{\lVert \xi e_j\rVert}= \lVert e_j - e_k \rVert = \sqrt{2} \geq \eta,
\]
and $\# E$ is not finite, so $H$ cannot be directionally limited.
\end{example}

\begin{example}
Let $\rho(r,\theta)=re^{i\theta}$ be the radial coordinate map of $\R^2$ and let $(u_n)_{n\geq 1}$
be a sequence in $[0,\pi/2]$ converging to $0$ ($(1/n)_{n\geq 1}$ does the job).
Define $X$ to be the image by $\rho$ of the compact set
\[
K := \big( [0,1] \times \{0, u_1, u_2, \dots \} \cap \{ (r,\theta)\in[0,1]\times[0,\pi/2] \mid r\leq \theta \} \big)
\cup \big([0,1]\times \{0\}\big),
\]
and endow $X$ with length distance induced by the Euclidean distance in $\R^2$.
The space $X$ is a metric tree rooted at the origin, here named $a$, and thus is $\CAT(0)$, of Hausdorff dimension $1$.
Moreover, $(X,d)$ is compact as every sequence in $X$ admits a convergent subsequence.
It is clear that $d$ cannot be directionally limited at $a$.
Indeed, consider the set $E=\{y\in X|d(a,y)<\xi\}$ and any pair of points $b,c\in E$ lying on different branches.
One has $d(x,c)=2d(a,c)$ and so
\[
\frac{d(x,c)}{d(a,c)} = 2 > \eta.
\]
But $\# E$ is not finite.
\end{example}

The next examples shows the necessity of a lower curvature bound to obtain a Besicovitch cover.

\begin{example}
If $(M^n,g)$ is a simply connected Riemannian manifold with nonpositive sectional curvature and covers a compact manifold
such that $M^n$ satisfies a Besicovitch covering Lemma then
$(M^n,g)$ is isometric to $\R^n$, see \cite[Theorem 1]{chiBesicovitch1991}.
\end{example}

\begin{example}
Let $\mathbb{H}$ be the Heisenberg group endowed with its Carnot-Carathéodory distance $d_{\mathrm{CC}}$.
The space $(\mathbb{H}, d_{\mathrm{CC}})$ is neither $\CAT(\kappa)$, because geodesics are not unique, nor $\CBB(\kappa)$.
In fact, Juillet proved in \cite{juilletHeisenberg2009} that $\mathbb{H}$ satisfies the \textit{measure contraction property}
$\operatorname{MCP}(0, 5)$ but never the Lott-Sturm-Villani condition $\CD(K,N)$.
Yet this is not sufficient for Besicovitch's covering lemma to hold, as shown by Korányi-Riemann in \cite[Part 1.4]{koranyiQuasiconformalHeisenberg1995}.
\end{example}

Concerning synthetic Ricci curvature bounds, note that in dimension 2, $\CBB(\kappa)$ and $\RCD(\kappa, 2)$ conditions
coincide \cite{petruninAlexandrovMeetsLottVillaniSturm2011, lytchakRicciCurvatureDimension2022}.
Thus $\RCD(\kappa, 2)$ spaces are directionally limited.
However, for higher dimensional $\RCD(\kappa, N)$ spaces the question remains open.

\section{Preliminaries}

We recall some elementary facts about metric geometry and CBB spaces.
One can find exhaustive introductions in \cite{buragoMetricGeometry2001} and \cite{alexanderAlexandrovGeometry2024} for example.

\subsection{Metric spaces and notations}

Recall that $(\cdot)^-=-\min(\cdot, 0)$.

Let $(X,d)$ be a metric space.
Distances are always assumed to be finite.
We denote by $B(x,r)$ (resp. $B[x,r]$) the open (resp. closed) ball centered at $x\in X$ with radius $r>0$.
Let $\mathbb{M}^n(\kappa)$ denote the simply connected $n$-dimensional manifold of constant curvature $\kappa$ and
write $\varpi^\kappa$ its diameter.
Let $x,y,z\in X$.
We denote by $\tilde\Delta^\kappa(xyz)$ the \textit{model triangle} associated to the points $x,y,z$ in $\mathbb{M}^2(\kappa)$.
The \textit{model angle}, denoted by $\modang^\kappa(x_y^z)$, is the inner angle at the vertex corresponding to $x$
in $\tilde\Delta^\kappa(xyz)$.
For any $\theta\in[0,\pi]$ and $a,b>0$, we denote by $\tilde\gamma^\kappa\{\theta; a,b \}$ the length of the side opposite to
the vertex of inner angle $\theta$ and of adjacent sides of length $a$ and $b$ in $\mathbb{M}^2(\kappa)$.
For any $x,y\in X$, the notation $[xy]$ denotes a choice of geodesic if it exists and $\geod_{[xy]}$ is its parametrization by arc length.

The \textit{modified sine function} $\sn^\kappa$ is defined, for any $\kappa\in\R$, by
\[
\sn^\kappa(t) :=
\begin{cases}
t, & \text{if }\kappa = 0, \\
\frac{1}{\sqrt{\kappa}} \sin(\sqrt{\kappa}\cdot t), & \text{if }\kappa >0, \\
\frac{1}{\sqrt{|\kappa|}} \sinh(\sqrt{|\kappa|}\cdot t), & \text{if }\kappa < 0,
\end{cases}
\]

\subsection{Curvature bounded below}

\begin{definition}
Let $\kappa\in\R$.
A metric space $(X,d)$ is said to have \emph{curvature bounded below by $\kappa$} (abbreviated $\CBB(\kappa)$)
if for every quadruple of points $p,x_1,x_2,x_3\in X$ the following condition holds:
\[
\modang^\kappa(p^{x_1}_{x_2}) + \modang^\kappa(p^{x_2}_{x_3}) + \modang^\kappa(p^{x_3}_{x_1}) \leq 2\pi,
\]
or one of the angles $\modang^\kappa(p^{x_i}_{x_j})$ is not defined.
The space $(X,d)$ is \emph{locally $\CBB(\kappa)$} if for every $p\in X$ there exists a neighborhood $U\ni p$ such that $U$
is $\CBB(\kappa)$.
More generally, $(X,d)$ is called \textit{CBB} provided it is $\CBB(\kappa)$ for any $\kappa$.
\end{definition}

Suppose that the model triangle $\tilde\Delta^\kappa(xyz)$ is defined and the geodesics $[xy]$ and $[xz]$ exist.
Then, if $(X,d)$ is $\CBB(\kappa)$, the map
\[
(t,t') \mapsto \modang^\kappa\left(x_{\geod_{[xy]}(t)}^{\geod_{[xz]}(t')}\right)
\]
is non-increasing with respect to each of its two variables \cite[Theorem 8.17]{alexanderAlexandrovGeometry2024}.

For any $(X,d)$ geodesic $\CBB(\kappa)$ space, the completion $(\bar X, \bar d)$ is a length $\CBB(\kappa)$ space
\cite{petruninGlobalizationNoncomplete2016}.
Moreover, if $X$ is locally compact and has Hausdorff dimension $n$, then $X$ is open in $\bar X$ and $\bar X$ also has dimension $n$.
As $\bar X$ is a complete length $\CBB(\kappa)$ space it is proper and geodesic \cite[Corollary 15.14]{alexanderAlexandrovGeometry2024}.

On a $n$-dimensional complete length CBB space $(X,d)$ the $n$-dimensional Hausdorff measure is a Radon measure with full support,
i.e.\ $0<\mathcal{H}^n(B(p,r))<\infty$ for any $p\in X, r>0$.
Every point of a dense subset \cite[Corollary 6.7]{buragoMetricGeometry2001} is contained in a neighborhood that is bi-Lipschitz equivalent to an open set of $\R^n$
\cite[Theorem 5.4]{buragoADAlexandrovSpaces1992} and, by Bishop inequality \cite[Theorem 10.6.8]{buragoMetricGeometry2001},
$\mathcal{H}^n(B(p,r))\leq v_{\kappa, n}(r)$.
By Stone's theorem $X$ is paracompact, by \cite[Theorem 5.1.27]{engelkingTopology1989} and by connectedness, $X$ is Lindelöf, and
by local compactness, this latter is equivalent to $\sigma$-compactness.
The Radon property follows by \cite[Theorem 2.18]{rudinAnalysis1966}.
Note that Lindelöf property is also equivalent to $X$ separable.

The separability argument applies verbatim under the milder hypothesis of a geodesic locally complete CBB space.

From Bishop-Gromov \cite[Thm 10.6.6]{buragoMetricGeometry2001} and \cite{sturmGeometryII2006},
we conclude that $\mathcal{H}^n$ is locally doubling:

\begin{theorem}
\th\label{thm:BishopGromov}
Let $(X,d)$ be a complete locally compact $\CBB(\kappa)$ space of dimension $n$.
Then, for any $0<r<R<L$,
\[
\frac{\mathcal{H}^n(B(p, R))}{\mathcal{H}^n(B(p,r))}
\leq \frac{v_{n,\kappa}(R)}{v_{n,\kappa}(r)}
\leq D(\kappa, n, L) \left( \frac{R}{r} \right)^n,
\]
where
\[
D(\kappa, n, L) =
\begin{cases}
1, & \text{if }\kappa \geq 0 \text{ or } n=1, \\
\cosh\left(L \sqrt{\frac{|\kappa|}{n-1}}\right)^{n-1}, & \text{if }\kappa<0.
\end{cases}
\]
\end{theorem}

\section{Proof}

Assume that $(\tilde X,\tilde d), \xi$ and $\eta$ are as in the hypothesis of \thref{thm:CBBDirLim}.
Let $(X, d)$ be the completion of $(\tilde X,\tilde d)$.
By the above discussion, $(X,d)$ is a complete geodesic proper $\CBB(\kappa)$ space of dimension $n$.
From now on we work exclusively in $(X,d)$.
Define
\[
\chi := \frac{\eta}{\cosh(\sqrt{\kappa^-}\xi)}.
\]

\begin{lemma}
\th\label{thm:distBdBelow}
Let $a\in X$ and $E = \{ e_1, \dots, e_m \}\subset B(a,\xi)\setminus\{a\}$ satisfy the hypothesis of \thref{def:DirLim}.
Let $\geod_{[ae_j]}:[0,t_j]\to X$ be a unit-speed geodesic from $a$ to $e_j$.
Then, for any $i\neq j$ and $0<s\leq\min(t_i, t_j)$,
\begin{equation}
\label{eq:traceDist}
d(\geod_{[ae_i]}(s), \geod_{[ae_j]}(s)) \geq \frac{4\chi}{\pi^2}s.
\end{equation}
\end{lemma}

\begin{proof}
Let $i\neq j$ be fixed assuming that $t_j\leq t_i$ and define $x:=\geod_{[ae_i]}(t_j)$.
Then $d(a,x)=t_j=d(a,e_j)$ and $d(x,e_i)=t_i-t_j$, so $x$ is admissible for the pair $(e_i,e_j)$ and
\[
d(x, e_j) \geq \eta t_j.
\]
Consider the triangle with vertex $a$ and opposite side $[xe_j]$.
In $\mathbb{M}^2(\kappa)$, for any $\alpha\in[0,\pi]$, it holds
\[
\tilde\gamma^\kappa\{ \alpha; t_j, t_j \} \leq \alpha \sn^\kappa(t_j),
\]
because the geodesic is shorter than the circle arc in polar coordinates.
One can easily show that $\sn^\kappa(t)\leq t \cosh(\sqrt{\kappa^-}t)$ for $t\geq0$.
Thus, we obtain
\[
\modang^\kappa(a_{b_j}^x) \geq \frac{\eta t_j}{\sn^\kappa(t_j)} \geq \chi.
\]

Now we fix a $0<s\leq t_j$.
By monotonicity of the model angle in CBB spaces,
\[
\modang^\kappa(a^{\geod_{[ae_j]}(s)}_{\geod_{[ae_i]}(s)}) \geq
\modang^\kappa(a^{\geod_{[ae_j]}(t_j)}_{\geod_{[ae_i]}(t_j)}),
\]
and so
\[
d(\geod_{[ae_i]}(s),\geod_{[ae_j]}(s))
= \tilde\gamma^\kappa\{ \modang^\kappa(a^{\geod_{[ae_i]}(s)}_{\geod_{[ae_j]}(s)}) ; s, s \}
\geq \tilde\gamma^\kappa\{ \chi ; s,s \}.
\]
It is well-known that $\sin(\alpha) \geq 2\alpha/\pi$ holds for any $\alpha\in[0, \pi/2]$.
Consider an isosceles triangle in $\mathbb{M}^2(\kappa)$ with angle $\chi$, and adjacent sides of length $s$.
We cut the triangle at its bisector from the vertex of angle $\chi$ and distinguish three cases:
\begin{enumerate}[wide, labelwidth=0pt, labelindent=0pt]
\item[($\kappa =0$)]
By law of sines,
\[
\tilde\gamma^\kappa\{\chi;s,s\} = 2s\sin\left(\frac{\chi}{2}\right) \geq \frac{2\chi}{\pi}s.
\]
\item[($\kappa>0$)]
By spherical law of sines,
\[
\sin\left( \sqrt{\kappa}\frac{\tilde\gamma^\kappa\{\chi; s,s\}}{2} \right)
= \sin(\sqrt{\kappa} s) \sin\left(\frac{\chi}{2} \right).
\]
Hence, as $\sqrt{\kappa}s \leq \sqrt{\kappa}\xi \leq \pi/2$,
\[
\begin{aligned}
\tilde\gamma^\kappa \{\chi; s,s\}
&= \frac{2}{\sqrt{\kappa}} \arcsin\left( \sin(\sqrt{\kappa} s) \sin\left(\frac{\chi}{2} \right) \right) \\
&\geq \frac{2}{\sqrt{\kappa}} \sin(\sqrt{\kappa} s) \sin\left(\frac{\chi}{2} \right) \\
&= 2 \sn^\kappa(s) \sin\left( \frac{\chi}{2} \right) \\
&\geq \frac{4}{\pi} s \sin\left( \frac{\chi}{2} \right) \\
&\geq \frac{4 \chi}{\pi^2} s.
\end{aligned}
\]
\item[($\kappa<0$)]
By hyperbolic law of sines and convexity of $\sinh$ on $[0,\infty)$,
\[
\begin{aligned}
\sinh\left(\sqrt{|\kappa|} \frac{\tilde\gamma^\kappa}{2}\right)
&= \sinh(\sqrt{|\kappa|} s) \sin\left( \frac{\chi}{2} \right) \\
&\geq \sinh\left( \sqrt{|\kappa|} s \sin\left(\frac{\chi}{2} \right) \right).
\end{aligned}
\]
Hence
\[
\tilde\gamma^\kappa\{\chi; s,s\} \geq 2s\sin\left( \frac{\chi}{2} \right)
\geq \frac{2\chi}{\pi} s.
\]
\end{enumerate}
From the three cases combined, we conclude that
\[
d(\geod_{[ae_i]}(s),\geod_{[ae_j]}(s)) \geq \frac{4\chi}{\pi^2} s. \qedhere
\]
\end{proof}

Now we want a packing number for the set $\{\geod_{[ae_i]}(s)\}$ given by \thref{thm:distBdBelow}.
One could hope to directly use a packing characterization of the finite dimension \cite[Theorem 15.13]{alexanderAlexandrovGeometry2024},
that is, for any $0<r\leq\varpi^\kappa, \varepsilon>0$ and $p\in X$,
\[
\pack_\varepsilon B_{\mathbb{M}^n(\kappa)}[r] \geq \pack_\varepsilon B(p,r).
\]
However, in \eqref{eq:traceDist} the chosen $s$ appears and packing number will still depend on it.
We thus need to estimate $\pack_\varepsilon B(p,r)$ independently of $s$.

\begin{lemma}[Bishop-Gromov packing]
\th\label{th:BGpack}
Let $a\in X$, $r>0$, $0<\delta\leq 2$, and let $E\subset B[a,r]$ be a finite set whose points have pairwise distances $\geq\delta r$.
Then
\[
\# E \leq D(\kappa, n, 4r) \left(\frac{9}{\delta}\right)^n.
\]
\end{lemma}

\begin{proof}
The proof follows a standard packing argument provided a measure is doubling.
The balls $\{B(z,\delta r/3)\}_{z\in E}$ are pairwise disjoint, because distinct centers satisfy $d(z,z')\geq\delta r>2\delta r/3$,
and they are contained in $B\big(a,(1+\delta/3)r\big)\subset B(a,2r)$ since $\delta\leq2$.
For each $z\in E$ we have $B(a,2r)\subset B(z,3r)$, and so, by \thref{thm:BishopGromov},
\[
\mathcal{H}^n\left(B\big(z,\tfrac{\delta r}{3}\big)\right)
\geq
\mathcal{H}^n\big(B(z,3r)\big)\frac{v_{\kappa,n}(\delta r/3)}{v_{\kappa,n}(3r)}
\geq
\mathcal{H}^n\big(B(a,2r)\big) D(\kappa, n, 4r)^{-1}\left(\frac{\delta}{9}\right)^n.
\]
Summing over $z\in E$, using disjointness and $0<\mathcal{H}^n(B(a,2r))<\infty$, we obtain
\[
\# E\cdot \mathcal{H}^n\big(B(a,2r)\big) D(\kappa, n, 4r)^{-1}\left(\frac{\delta}{9}\right)^n
\leq\mathcal{H}^n\big(B(a,2r)\big),
\]
yielding the desired bound.
\end{proof}

Let $a\in X$ and let $E\subset B(a,\xi)\setminus\{a\}$ satisfy the hypothesis of \thref{def:DirLim}.
It suffices to prove the bound when $E$ is finite, say $E=\{e_1,\dots,e_m\}$.
Let $\gamma_i$ be as in \thref{thm:distBdBelow}. By \eqref{eq:traceDist}, the points
\[
p_i:=\gamma_i(s)\in B[a,s],\qquad i=1,\dots,m,
\]
have pairwise distances $\geq 4\chi s/\pi^2$.
Applying \thref{th:BGpack} with $r=s$ and $\delta=4\chi/\pi^2\leq2$ yields
\[
m\leq D(\kappa, n, 4\xi)\left(\frac{9\pi^2}{4\chi}\right)^n.
\]
With the constants above and
\[
\zeta = \left\lceil D(\kappa, n, 4\xi) \left( \frac{9\pi^2\cosh(\sqrt{\kappa^-}\xi)}{4\eta} \right)^n \right\rceil,
\]
the distance $d$ is $(\xi,\eta,\zeta)$-directionally limited at $X$.
As $(\tilde X,\tilde d)$ is geodesic, $\tilde d=d$ on $\tilde X$ and so $\tilde d$ is $(\xi,\eta,\zeta)$-directionally limited
at $\tilde X$.
This achieves the proof of \thref{thm:CBBDirLim}.

\section{Acknowledgements}

This research was funded by the Deutsche Forschungsgemeinschaft (DFG).
The author is grateful to his advisor Batu Güneysu for his constant support.

\bibliographystyle{abbrvnat}
\bibliography{references}

\end{document}